\documentclass[11pt,a4paper]{article}

\usepackage{theorem,enumerate}
\usepackage{amsmath,latexsym,amssymb,amsfonts}
\usepackage{eucal}
\usepackage{mathrsfs}
\usepackage{array}
\usepackage{footmisc}
\usepackage{graphicx, color}
\usepackage{caption,subcaption}

\newtheorem{theorem}{Theorem} \rm
\newtheorem{lemma}[theorem]{Lemma}

\newtheorem{definition}[theorem]{Definition}

\numberwithin{theorem}{section}

\newcommand{\per}{{\rm per}}

\usepackage{color}

\newcommand{\qed}{\hfill \ensuremath{\Box}}

\begin{document}
\title{\bf Dense Eulerian graphs are $(1, 3)$-choosable}

\author{
	Huajing Lu\thanks{Department of Mathematics, Zhejiang Normal University, China.
		e-mail: {\tt huajinglu@zjnu.edu.cn
		}} \thanks{College of Basic Science, Ningbo University of Finance and Economics, China.}\and
	Xuding Zhu\thanks{Department of Mathematics, Zhejiang Normal University, China.
		e-mail: {\tt  xdzhu@zjnu.edu.cn
		}} \thanks{Grant Numbers:  NSFC 11971438,12026248, U20A2068.}
	}

\maketitle

\begin{abstract}
A graph $G$ is total weight $(k,k')$-choosable if for any total list assignment $L$ which assigns to each vertex $v$ a set $L(v)$ of $k$ real numbers, and each edge $e$ a set $L(e)$ of $k'$ real numbers, there is a proper total $L$-weighting, i.e., a mapping $f: V(G) \cup E(G) \to \mathbb{R}$ such that for each $z \in V(G) \cup E(G)$, $f(z) \in L(z)$, and for each edge $uv$ of $G$,  $\sum_{e \in E(u)}f(e)+f(u) \ne \sum_{e \in E(v)}f(e) + f(v)$.  
This paper proves that if $G$ decomposes into   complete graphs of odd order, then $G$ is total weight $(1,3)$-choosable. As a consequence, every Eulerian graph $G$ of large order and with minimum degree at least $0.91|V(G)|$ is total weight $(1,3)$-choosable. We also prove that any graph $G$ with minimum degree at least $0.999|V(G)|$ is total weight $(1,4)$-choosable.
\end{abstract}

\section{Introduction}
Assume $G=(V,E)$ is a graph with vertex set $V=\{1, 2, \dots, n\}$.  Each edge $e \in E$ of $G$ is   $2$-subset $e=\{i, j\}$ of  $V$. A {\em total weighting} of $G$ is a mapping $\phi$: $V\cup E\rightarrow \mathbb{R}$. A total weighting $\phi$ is {\em proper} if for any edge $\{i, j\}\in E$,
$$\sum_{e\in E(i)}\phi(e)+\phi(i)\ne \sum_{e\in E(j)}\phi(e)+\phi(j).$$

A proper total weighting $\phi$ with $\phi(i)=0$ for all vertices $i$ is also called a {\em vertex coloring edge weighting}. A   vertex coloring  edge weighting  of   $G$ using weights $\{1, 2, \dots, k\}$ is called a {\em vertex coloring $k$-edge weighting}. Note
that if $G$ has an isolate edge, then $G$ does not admit a vertex coloring edge weighting. We say a graph is {\em nice} if it does not contain any isolated edge.

Karo\'{n}ski,  {\L}uczak and  Thomason  \cite{KLT2004} conjectured that every nice graph has a vertex coloring $3$-edge weighting. This conjecture received considerable attention \cite{Add2005,Add2007,Grytczuk2021,KKP10,TWZ2016,WY2008,Zhong}, and it is known as the $1$-$2$-$3$ conjecture. The best result on $1$-$2$-$3$ conjecture so far was obtained  by Kalkowski, Karo\'{n}ski and   Pfender \cite{KKP10},  who proved that every nice graph has a vertex coloring $5$-edge weighting.

The list version of edge weighting of graphs was introduced by 
Bartnicki, Grytczuk and   Niwczyk  \cite{BGN09}. The list version of total weighting of graphs was introduced independently by Przyby\l o and  Wo\'{z}niak in \cite{PW2010} 
and by Wong and Zhu in \cite{WZ11}. Let $\psi: V\cup E\to \mathbb{N}^+$. 
A $\psi$-list assignment of $G$ is a mapping $L$ which assigns 
to $z\in V\cup E$ a set $L(z)$ of $\psi(z)$ real numbers. 
Given a total list assignment $L$, a {\em proper $L$-total weighting} 
is a proper total weighting $\phi$ with $\phi(z)\in L(z)$ for all $z\in V\cup E$. We say $G$ is {\em total weight $\psi$-choosable} ($\psi$-choosable for short) if for any $\psi$-list assignment $L$, 
there is a proper $L$-total weighting of $G$. We say $G$ is 
{\em total weight $(k, k')$-choosable} ($(k, k')$-choosable for short)
if $G$ is $\psi$-total
weight choosable, where  $\psi(i)= k$ for $i \in V (G)$ and  $\psi(e) = k'$ for $e\in E(G)$.

List version of edge weighting also received a 
lot of attention \cite{BGN09,Cao,CDWZ2017,DDWWWYZ2019,PW2011,TWZ2016,Wong2021,WZ11,WZ2012,Zhu}.  
As strengthenings of the 1-2-3 conjecture, 
it was conjectured  in \cite{WZ11} that every nice graph is $(1, 3)$-choosable.  
A weaker conjecture was also proposed in \cite{WZ11}, which asserts that there is a constant $k$ such that every nice graph is $(1, k)$-choosable.
This weaker conjecture was recently confirmed by Cao \cite{Cao}, who proved that every
nice graph is $(1, 17)$-choosable.  This result was improved in \cite{Zhu}, where it was shown that
  every nice graph is $(1, 5)$-choosable.

Given a graph $G$ and a family of graphs $\mathcal{H}$, we say that $G$ has an $\mathcal{H}$-decomposition, if the edges of $G$ can be  partitioned into the edge sets of copies of graphs from $\mathcal{H}$. In particular, a triangle decomposition of $G$ is a partition of $E(G)$ into triangle, and for a given graph $H$, an $H$-decomposition of $G$ partitions $E(G)$ into subsets, each inducing a copy of $H$.   The following is the main result of this paper.

\begin{theorem}
	\label{thm-main}
	If $E(G)$ can be decomposed into cliques of odd order, then $G$ is   $(1, 3)$-choosable.
\end{theorem}

As a consequence of Theorem \ref{thm-main}, we prove the following result.

\begin{theorem}
\label{dense}
If $G$ is an $n$-vertex Eulerian graph with minimum degree at least $0.91 n$ and $n$ sufficiently large, then $G$ is $(1,3)$-choosable. 
\end{theorem}

In \cite{Zhong},   Zhong confirmed the $1$-$2$-$3$ conjecture for graphs that can be edge-decomposed into cliques of order at least 3. As a consequence of this result, it was proved in \cite{Zhong} that the 1-2-3 conjecture holds for every $n$-vertex graph with minimum degree at least $0.99985n$, where $n$ is sufficiently large.

Our result is a list version of Zhong's result, but with one degree    restriction:  $E(G)$ needs to be decomposed into complete graphs of odd order.  Hence we can only show that dense Eulerian graphs are $(1,3)$-choosable. For general dense graphs, we prove the following result:

\begin{theorem}
\label{thm-dense2}
If $G$ is an $n$-vertex  graph with minimum degree at least $0.999 n$, then $G$ is $(1,4)$-choosable. 
\end{theorem}

\section{Algebraic total weight choosability}

The proof of Theorem \ref{thm-main} applies Combinatorial Nullstellensatz \cite{nullstellensatz} and uses the tools introduced in \cite{Cao} that was further developed in \cite{Zhu}.

 Given a graph $G=(V,E)$, let 
 $$\tilde{P}_G(\{x_z: z \in V \cup E\}) = \prod_{\{i,j\} \in E, i < j}\left(  \left(\sum_{e \in E(i)} x_e+ x_i\right) - \left(\sum_{e \in E(j)} x_e+ x_j\right)\right).$$
 Assign a real number $\phi(z)$ to the variable $x_z$, and view
 $\phi(z)$ as the weight of $z$.
 Let $\tilde{P}_G( \phi  )$ be the evaluation of the
 polynomial at $x_z = \phi(z)$. Then $\phi$ is a proper total
 weighting of $G$ if and only if $\tilde{P}_G( \phi) \ne 0$. Thus the problem of finding a proper $L$-total weighting of $G$ (for a given total list assignment $L$) is equivalent to find a non-zero point of the polynomial $\tilde{P}_G(\{x_z: z \in V \cup E\})$ in the grid $\prod_{z \in V \cup E} L(z)$.
 
 Combinatorial Nullstellensatz gives a sufficient condition for the existence of a non-zero point in a given grid. 
 
We denote by $\mathbb{N}$ and $\mathbb{N}^+$ the set of non-negative integers and the set of positive integers, respectively. For   $m,n \in \mathbb{N}^+$, let $\mathbb{C}[x_1, x_2, \ldots, x_n]_m$ be the vector space of homogeneous polynomials of degree $m$ in variables $x_1, \ldots, x_n$ over the   field $\mathbb{C}$ of complex numbers. We denote by $M_{n,m}(\mathbb{C})$ the set of $n \times m$ matrices with entries in $\mathbb{C}$.

For a finite set $E$, let 
$$\mathbb{N}^E = \{K: E \to \mathbb{N}  \},  \mathbb{N}^E_m = \{K \in \mathbb{N}^E:  \sum_{e \in E} K(e)=m\}.$$ 
Let $$\mathbb{N}^E_{(k^-)} = \{K \in \mathbb{N}^E: K(e) \le k,  \forall e \in E\}.$$
For 
$K \in \mathbb{N}^E$, let 
$$x^K = \prod_{e \in E}x_e^{K(e)}.$$
Let 
$$K!=\prod_{e \in E}K(e)!.$$

Given a  polynomial $P$, we denote the coefficient of the monomial  $x^K$ in the expansion of $P$ by 
$${\rm coe}(x^K, P).$$ 
 Let $${\rm mon}(P)=\{x^K: {\rm coe}(x^K,P) \ne 0\}.$$  
 
 It follows from Combinatorial Nullstellensatz that if $\prod_{z \in V \cup E}x_z^{K(z)} \in {\rm mon}(\tilde{P}_G)$, and $|L(z)| \ge K(z)+1$ for some $K \in \mathbb{N}^{E \cup V}$, then $G$ has a proper total $L$-weighting. 
 
 \begin{definition}
 	 A  graph is said to be {\em algebraic total weight $(k,k')$-choosable} ({\em algebraic  $(k,k')$-choosable} for short) if $x^K=\prod_{z \in V \cup E}x_z^{K(z)} \in {\rm mon}(\tilde{P}_G)$
 	 for some $K \in \mathbb{N}^{E \cup V}_{|E|}$ with $K(i) < k$ for each vertex $i$ and $K(e) < k'$ for each edge $e$.
 \end{definition}
 
  This paper is interested in $(1, b+1)$-choosability of graphs. 
  That is to show that for some $K \in \mathbb{N}_{(b^-)}^E$,  $x^K \in {\rm mon}(  \tilde{P}_G)$. For this purpose, we omit the variables $x_i$ for $i \in V$ and consider the following polynomial:

 $$P_G(\{x_e: e \in E\}) = \prod_{\{i,j\} \in E, i < j}\left(\sum_{e\in E(i)}x_{e} - \sum_{e\in E(j)}x_{e}\right).$$ We say $K$ is {\em sufficient for $G$} if there exists $K' \in \mathbb{N}^E$ such that $K' \le K$ and $x^{K'} \in {\rm mon}(P_G)$. 
 
For a   matrix $A=(a_{ij})_{m \times n}$, define polynomial 
$$F_A(x_1, \ldots, x_n) = \prod_{i=1}^m \sum_{j=1}^n  a_{ij}x_j.$$
Given a graph $G=(V,E)$,  
let 
$C_G=(c_{ee'})_{e,e'\in E}$, where for $e=\{i,j\} \in E, i < j$, 
\[
c_{ee'} = \begin{cases} 1, &\text{ if $e'$ is adjacent with $e$ at $i$}, \cr
-1, &\text{ if $e'$ is adjacent with $e$ at $j$}, \cr
0, &\text{ otherwise.}
\end{cases}
\]
Let 
$A_G=(a_{e i})_{e \in E, i \in V}$, where for $e=\{s,t\} \in E, s < t$, 
\[
a_{ei} = \begin{cases} 1, &\text{ if $i=s$}, \cr
-1, &\text{ if $i=t$}, \cr
0, &\text{ otherwise.}
\end{cases}
\]
and  
$B_G=(b_{ei})_{e \in E, i \in V}$, where   
\[
b_{ei} = \begin{cases} 1, &\text{ if $i$ is incident to $e$}, \cr
0, &\text{ otherwise.}
\end{cases}
\]
It is easy to verify (cf. \cite{Cao}) that $$P_G=F_{C_G}, C_G = A_G (B_G)^T.$$

For a square matrix $A=(a_{ij})_{n \times n}$,   the {\em permanent} $\per(A)$ of $A$ is defined as  
$$\per(A) = \sum_{\sigma}\prod_{i=1}^n a_{ i \sigma(i)},$$ where the summation is over all permutations $\sigma$ of 
$\{1,2,\ldots, n\}$.
For $A \in M_{m,n}(\mathbb{C})$, 
for $K   \in \mathbb{N}^n$ and $K'   \in \mathbb{N}^m$,   $A(K)$ denotes the matrix whose columns consist of $K(i)$ copies of the $i$th column of $A$, and $A[K']$ denotes the matrix whose rows consist  of $K'(i)$ copies of the $i$th row of $A$.

It is known \cite{alontarsi1989,WZ11,WZ2017} and easy to verify that  for any $A \in M_{m,n}(\mathbb{C})$ and $K   \in \mathbb{N}_m^n$,
\begin{eqnarray} 
&&{\rm coe}(x^K, F_A) = \frac{1}{ K!}\per(A(K)). \label{eqn1}
\end{eqnarray}
 
As $ C_G(K) = A_GB_G[K]^T$,  
\begin{eqnarray} 
\label{eqn1b} {\rm coe}(x^K,P_G)=  \frac{1}{ K!}\per(C_G(K)) = \frac{1}{ K!}\per(A_G B_G[K]^T). 
\end{eqnarray}

\section{Proof of Theorems \ref{thm-main}, \ref{dense} and \ref{thm-dense2}}


  Consider the vector space   of homogeneous polynomial of degree $|E|$ in $\mathbb{C}[x_e: e \in E]$. An {\em inner product} in this space is defined as 
  $$\langle f, g\rangle = \sum_{K \in \mathbb{N}^n_m} K! {\rm coe}(x^K,f)\overline{ {\rm coe}(x^K,g)}.$$

  By (\ref{eqn1b}), we are interested in calculating the permanent of matrix of the form $AB^*$, where $B^*$ is the conjugate transpose of $B$. The following lemma was proved in \cite{Cao}.
  
  \begin{lemma}
  For matrices $A, B \in M_{n,m}(\mathbb{C})$, $$\per(A B^*) = \langle f_A, f_B \rangle.$$ 
  \end{lemma}
  
  So to prove $K \in \mathbb{N}^E$ is sufficient for $G$, it suffices to show that $$\langle f_{A_G}, f_{B_G[K]} \rangle\ne0.$$ 
  On the other hand, if $E$ is the edge set of $G$, then it follows from the definitions that 
\begin{center}
$f_{A_G}=\prod\limits_{e=\{i, j\}\in E, i<j}(x_i-x_j)$, and $f_{B_G[K]}=\prod\limits_{e=\{i, j\}\in E, i<j}(x_i+x_j)^{K(e)}$.
\end{center}
 
\begin{definition} \label{linearspace} 
For $K\in\mathbb{N}^E$, let $W_E^K$ be the complex vector space spanned by $$\{\prod\limits_{e=\{i, j\}\in E, i<j}(x_i+x_j)^{K'(e)}: K' \le K\}.$$
\end{definition}

Thus we have the following lemma, which was proved in \cite{Cao}.

\begin{lemma} \label{equivalent}
	Assume $G$ is a  graph with edge set $E$ and $K\in\mathbb{N}^E$. Then $K$ is sufficient for $G$ if and only if $\left< F, f_{A_G}\right>\ne0$ for some $F \in W_E^K$.
\end{lemma}

The following lemma is an easy observation, but it is the key tool for proving the main results of this paper.

\begin{lemma}
\label{lem-key}
If $f_{A_G} \in W_E^K$ for some $K \in \mathbb{N}_{(b^-)}^E$, then $G$ is algebraic $(1, b+1)$-choosable.
\end{lemma}
\begin{proof}
Assume $f_{A_G} \in W_E^K$. As $f_{A_G} \ne 0$, we have $\langle f_{A_G}, f_{A_G} \rangle > 0$. By Lemma \ref{equivalent}, $K$ is sufficient for $G$. As $K \in \mathbb{N}^E_{b^-}$, i.e., $K(e) 
\le b$ for all edges $e$, we conclude that $G$ is algebraic $(1, b+1)$-choosable. \qed
\end{proof}

As an example, consider a triangle $T$ with vertex set $\{i,j,k\}$. 
By definition, $f_{A_T} = (x_i-x_j)(x_j-x_k)(x_i-x_k)$. To prove that $f_{A_T} \in W_E^K$, we need to express $f_{A_T}$ as a polynomial in $(x_i+x_j), (x_j+x_k), (x_i+x_k)$ in such a way that for each edge $e$, say for $e=\{x_i, x_j\}$, the term $(x_i+x_j)$ occurs in the expression at most $K(e)$ times. 
We can write $f_{A_T}$ as 
$$f_{A_T}= ((x_i+x_k)-(x_j+x_k))((x_i+x_j)-(x_i+x_k))((x_i+x_j)-(x_j+x_k)).$$
It is easy to check that for each edge, say for $e=
\{x_i,x_j\}$, the term $(x_i+x_j)$ occurs twices in the expression above. Thus $f_{A_T} \in W_E^K$, and $K(e)=2$ for each edge $e$ of $T$. 

To express $f_{A_G}$ as a polynomial in $\{x_i+x_j: \{i,j\} \in E\}$, it suffices to express, for each edge $\{i,j\} \in E$, the term $(x_i-x_j)$ as a linear combination of terms $\{ (x_{i'}+x_{j'}): \{i',j'\} \in E\}$. This is done by choosing an even length path connecting vertices  $i$ and $j$ (see the proof of Lemma \ref{lem-cover} below). 

\begin{definition}
\label{def-pathcover}
Assume $G=(V,E)$ is a graph. A {\em path covering family} of $G$ is a family $\mathcal{P}$ of paths, that consists of, for each edge $e=\{i,j\}$,  an even length path $P_e$ connecting $i$ and $j$. 
\end{definition}

For a subgraph $H$ of $G$, $K_H: E \to \mathbb{N}$ is the characteristic function of $E(H)$, i.e., $K_H(e) = 1$ if $e \in E(H)$ and $K_H(e) = 0$ otherwise. For a multi-family $\mathcal{F}$ of subgraphs of $G$, 
$$K_{\mathcal{F}} = \sum_{H \in \mathcal{F}}K_H.$$

\begin{lemma}
\label{lem-cover}
If $G$ has a path covering family $\mathcal{P}$ with $K_{\mathcal{P}}(e) \le b$ for each edge $e$, then $G$ is algebraic $(1, b+1)$-choosable.
\end{lemma}
\begin{proof}
Assume $\mathcal{P}$ is a path covering family with $K_{\mathcal{P}}(e) \le b$ for each edge $e$. Assume $e=\{i,j\}$ is an edge of $G$, and $P_e = (i_0,i_1, \ldots, i_{2k})$ is an even length path connecting $i$ and $j$, i.e., $i_0=i$ and $i_{2k} = j$. Then 
$$x_i-x_j = \sum_{l=0}^{2k-1} (-1)^l (x_{i_l}+x_{i_{l+1}}) \in W_E^{K_{P_e}}.$$
Hence $$f_{A_G} = \prod_{ \{i,j\} \in E} (x_i-x_j)  \in W_E^{K_{\mathcal{P}}}.$$
As $K_{\mathcal{P}}(e) \le b$ for each edge $e$, we have $f_{A_G} \in W_E^K$ and $K \in \mathbb{N}^E_{b^-}$. By Lemma \ref{lem-key}, $G$ is algebraic $(1, b+1)$-choosable. \qed
\end{proof}

The following lemma follows easily from the definitions and its proof is omitted.

\begin{lemma}
\label{lem-decomp}
If $G$ decomposes into graphs $H_1,H_2, \ldots, H_q$, and each $H_i$ has a path covering family $ \mathcal{P}_i$ with $K_{\mathcal{P}_i} \in W_{E(H_i)}^{K_i}$ and $K_i \in \mathbb{N}^{E(H_i)}_{(b^-)}$, then $\mathcal{P} = \cup_{i=1}^q \mathcal{P}_i$ is a path covering family of $G$, $K_{\mathcal{P}} \in W_E^K$ and $K = \sum_{i=1}^q K_i \in \mathbb{N}^E_{(b^-)}$.  \qed
\end{lemma}

\bigskip
\noindent
{\bf Proof of  
Theorem \ref{thm-main}}: By Lemmas \ref{lem-cover} and \ref{lem-decomp}, it suffices to show that each complete graph $K_n$ of odd order has a path covering family $\mathcal{P}$ with $K_{\mathcal{P}} \in \mathbb{N}^E_{(2^-)}$. Assume $K_n$ has vertex set $\{1,2,\ldots, n\}$. For each edge $e=\{i,j\} \in E(K_n)$, where $i < j$,  
let 
 \[
 t_{i,j}  =
\begin{cases} 
i+\frac{j-i}{2},  & \mbox{if }j-i\mbox{ is even} \\
j+\frac{n-(j-i)}{2} \pmod n, & \mbox{if }j-i\mbox{ is odd},
\end{cases}
\]
and let $P_e = (i, t_{i,j}, j)$.  Then
$\mathcal{P} = \{P_e: e \in E(K_n)\}$ is a path covering family of $K_n$. For 
each edge $\{i,j\}$ of $K_n$, let $e_{i,j} = \{i, 2j-i\}$ and $e'_{i,j} = \{j, 2i-j\}$ (where calculations are modulo $n$), it is easy to verify that $\{i,j\}$ is contained in $P_{e_{i,j}}$ and $P_{e'_{i,j}}$. So each edge of $K_n$ is contained in two paths in $\mathcal{P}$, i.e., 
$K_{\mathcal{P}}(e) =2$ for each edge $e$ of $K_n$.   
  This completes the proof of Theorem \ref{thm-main}. \qed   
 
 For a graph $G$, let ${\rm gcd}(G)$  be the largest integer
dividing the degree of every vertex of $G$.  
We say that $G$ is $F$-divisible 
if $|E(G)|$ is divisible by
$|E(F)|$ and ${\rm gcd}(G)$ is divisible by ${\rm gcd}(F)$.

\bigskip
 \noindent
 {\bf Proof of Theorem \ref{dense}}
 
 The following result  was proved in \cite{Barber}:

\begin{theorem} 
\label{thm-triangle}
For every $\epsilon > 0$, there is an integer $n_0$ such that if $G$ is a triangle-divisible graph 
of  order $n \ge n_0$ and minimum degree at least $(0.9 + \epsilon)n$, then $G$ has a  triangle decomposition.
\end{theorem}

Assume $G$ is an $n$-vertex Eulerian graph of minimum degree $\delta(G) > (0.9 + \epsilon)n$. By Theorem \ref{thm-main}, it suffices to show that $G$ decomposes into complete graphs of odd order.

Assume $|E(G)| \equiv i \pmod{3}$, where $i \in \{0,1,2\}$. Let $H_1, \ldots, H_i$ be vertex disjoint $5$-cliques in $G$. Then $G'=G- \cup_{j=1}^iE(H_j)$ is triangle divisible and $\delta(G') \ge \delta(G)-4 \ge(0.9+\epsilon') n$. By Theorem \ref{thm-triangle}, $G'$ is triangle decomposible. Hence $G$ decomposes into complete graphs of odd order.   This completes the proof of Theorem \ref{dense}.

\bigskip

\noindent
{\bf Proof of Theorem \ref{thm-dense2}}:

\begin{lemma}
\label{lem-3}
 Let $H=(V,E)$ be the graph shown in Figure \ref{fig1}. Then $H$ has a path covering family $\mathcal{P}$ with $K_{\mathcal{P}} \in \mathbb{N}^E_{(3^-)}$.  
\end{lemma}
\begin{proof}
We denote by $T_1=(1,2,4), T_2=(2,3,5)$ the two edge disjoint triangles in $H$. 
For each triangle $T_i$, let $\mathcal{P}_i$ be the path covering family with $K_{\mathcal{P}_i} \in \mathbb{N}^{E(T_i)}_{(2^-)}$. For the edge $e=\{1,3\}$ which is not contained in the 2 triangles, let $P_e = (1,2,3)$. Then 
$$\mathcal{P} = \cup_{i=1}^4 \mathcal{P}_i \cup \{P_e\}$$
is a path covering family of $H$ with $K_{\mathcal{P}} \in \mathbb{N}^E_{(3^-)}$. 
This completes the proof of Lemma \ref{lem-3}.
\qed
\end{proof}

To prove Theorem \ref{thm-dense2}, we need the following   theorem proved in  \cite{Barber}:

\begin{theorem} 
\label{thm-H}
For every $\epsilon > 0$, there is an integer $n_0$ such that if $G$ is an $H$-divisible graph 
of  order $n \ge n_0$ and minimum degree at least $(1-1/t+ \epsilon)n$, where $t \max \{16 \chi(H)^2(\chi(H)-1)^2, |E(H)|\}$, then $G$ has an  $H$-decomposition.
\end{theorem}

Assume $G$ is a graph of large order and 
with minimum degree $\delta(G) \ge 0.999 |V(G)|$. If $|E(H)|$ divides $|E(G)|$,  then $G$ decomposes into copies of $H$ and Theorem \ref{thm-dense2} follows from Lemma \ref{lem-cover}. Otherwise, the same argument as the proof of Theorem \ref{dense} shows that $G$ can be decomposed into at most 6 copies of triangles and copies of $H$, and hence again 
Theorem \ref{thm-dense2} follows from Lemma \ref{lem-cover}. \qed

\begin{figure}[!htb]
 	\centering
 	 	\includegraphics[scale=0.5]{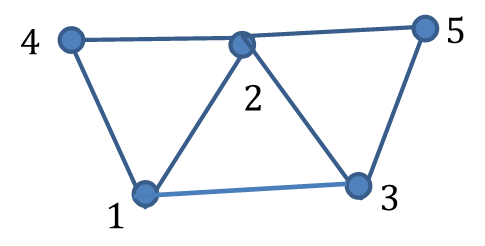}
 	\caption{The graph $H$. }\label{fig1}
 \end{figure}

\end{document}